\documentclass[12pt]{amsart}
\usepackage{graphicx}
\newtheorem{thm}{Theorem}[section]
\newtheorem{lem}[thm]{Lemma}

\newtheorem{prop}[thm]{Proposition}

\theoremstyle{definition}
\newtheorem{defn}[thm]{Definition}

\theoremstyle{remark}

\begin{document}

\title[Pandiagonal Type-$p$ Franklin Squares]{Pandiagonal Type-$p$ Franklin Squares}
\author{John Lorch}
\address{Department of Mathematical Sciences\\ Ball State University\\Muncie, IN  47306-0490}
\email{jlorch@bsu.edu}
\date{December 27, 2017}
\begin{abstract}
For prime $p$ we define magic squares of order $kp^3$, called type-$p$ Franklin squares, whose properties specialize to those of classical Franklin squares in the case $p=2$. We construct type-$p$ Franklin squares in prime-power orders.
\end{abstract}
\maketitle

\def\a{{\bf a}}
\def\MF{M^{2\times 2}(\F )}
\def\rtimes{{\times\!\! |}}
\def\rk{{\rm rank}}
\def\A{{\mathcal A}}
\def\reg{{\mathcal R}}
\def\P{{\mathcal P}}
\def\End{{\rm End}}
\def\F{{\mathbb F}}
\def\calF{{\mathcal F}}
\def\GL{{\rm GL}}
\def\L{{\mathcal L}}
\def\R{{\mathbb R}}
\def\N{{\mathbb N}}
\def\C{{\mathbb C}}
\def\Z{{\mathbb Z}}
\def\B{{\mathbb B}}
\def\G{{\mathcal G}}
\def\g{\mathfrak g}
\def\OOA{{\rm OOA}}
\def\OA{{\rm OA}}
\def\S{{\mathcal S}}
\def\SA{{\rm SA}}
\def\SD{{\rm SD}}
\def\TD{{\rm TD}}
\def\Q{{\mathcal Q}}
\def\pt#1{{\langle  #1 \rangle}}
\def\ds{\displaystyle}
\def\v{{\bf v}}
\def\w{{\bf w}}

\makeatletter
\def\Ddots{\mathinner{\mkern1mu\raise\p@
\vbox{\kern7\p@\hbox{.}}\mkern2mu
\raise4\p@\hbox{.}\mkern2mu\raise7\p@\hbox{.}\mkern1mu}}
\makeatother

\section{Introduction}

\subsection{Purpose, Briefly Stated}
For prime $p$ we define magic squares of order $kp^3$, called type-$p$ Franklin squares, whose properties specialize to those of classical Franklin squares in the case $p=2$. We construct such squares in prime power orders. Our construction is motivated by a relationship, first noted in \cite{pL06} and further explored in \cite{rN16}, between classical most-perfect magic squares of triply even order and pandiagonal classical Franklin squares.

\subsection{Franklin Squares}
Classical {\bf Franklin squares} are natural semi-magic squares of doubly even order  first constructed by Benjamin Franklin in the mid 1730's (two in order 8, one in order 16) to fend off  boredom while clerking in the Pennsylvania Assembly. They have the following additional magic properties:
\begin{itemize}
\item[(i)] Half-rows and half-columns add to half of the magic sum.
\item[(ii)] The symbols in any $2\times 2$ subsquare formed from consecutive rows and columns (allowing toric wraparound) sum to $\displaystyle \frac{p^2(n^2-1)}{2}$.
\item[(iii)] Entries in each set of {\bf bent diagonals} add to the magic sum. Bent diagonals come in four varieties: up, right, down, and left. An up-diagonal is formed by half of a broken main diagonal (allowing vertical wraparound) beginning at the left edge of the square, together with its reflection across the vertical midline. The right, down, and left varieties are obtained from the up-diagonal locations by $90^\circ$, $180^\circ$, and $270^\circ$ clockwise rotations of the ambient square, respectively.
\end{itemize}
Franklin's famous order-8 square is shown in Figure \ref{f:franklinex}.

\begin{figure}[h]
{\tiny
$$
\begin{array}{|c | c | c | c | c | c | c | c|}
\hline
51 & 60 & 3 & 12 & 19 & 28 & 35 & 44  \\
\hline
13  & 2 & 61 & 50 & 45 & 34 & 29 & 18 \\
\hline
52 & 59 & 4 & 11 & 20 & 27 & 36 & 43  \\
\hline
 10 & 5 & 58 & 53& 42 & 37 & 26 & 21  \\
\hline
 54 & 57 & 6 & 9 & 22 & 25 & 38& 41  \\
\hline
 8 &  7& 56 & 55 & 40& 39 & 24 & 23  \\
\hline
49 & 62 & 1 & 14 & 17 & 30 & 33 & 46  \\
\hline
15 & 0 & 63 & 48 & 47 & 32 & 31& 16 \\
\hline
    \end{array}
\qquad
\begin{array}{|c | c | c | c | c | c | c | c|}
\hline
\ & \ & \ & 12 & 19 & \ & \ & \  \\
\hline
13  & \ & \ & 50 & 45 & \ & \ & 18 \\
\hline
\ & 59 & \ & \ & \ & \ & 36 & \  \\
\hline
 \ &\ & 58 & \ & \ & 37 & \ &  \  \\
\hline
 \ & \ & \ & 9 & 22 & \ & 38& \  \\
\hline
 \ &  \ & \ & \ & \ & \ & 24 & \  \\
\hline
49 & 62 & 1 & 14 & \ & \ & 33 & \  \\
\hline
\ & \ & \  & \ & \ & \ & 31 & \ \\
\hline
    \end{array}
   $$}
    \caption{Left: Franklin's famous order-8 square with symbols $0$ through $63$.  Right: An indication of its properties. Numbers $13, 59,\dots ,36, 18$ form an up-diagonal.}
    \label{f:franklinex}
\end{figure}

Investigation of classical Franklin squares largely fits into three categories. The first is historical:
Franklin's method of constructing his squares remains unknown. His correspondence makes only brief mention of them, including a lament concerning the time he wasted in such activities. Pasles' article \cite{pP01} contains a thorough historical account of Franklin's squares and a survey of methods he may have used to construct them. The most plausible of these methods appears to be the one conjectured in \cite{cJ71}. Another category is existential: The definition of classical Franklin squares allows for doubly even orders, but the only Franklin squares that have been discovered thus far are of triply even order. Franklin squares exist in orders $8k$ for each $k\in\Z ^+$ (e.g., \cite{cJ71} and \cite{rN16}). Meanwhile, there are no Franklin squares of order $4$ or $12$ (see \cite{cH07}), and the existential question is unresolved for other orders of the form $8k+4$.
The third category concerns construction and enumeration: One example is \cite{mA04}, in which Hilbert bases for polyhedral cones are used to place an upper bound on the number of Franklin squares. Another example is \cite{pL06}, in which an involution on arrays is used to define an injection from the set of most-perfect squares of order $8$ to the set of pandiagonal Franklin squares of order $8$, thus giving a reasonable lower bound on the number of order-$8$ Franklin squares. Importantly, this latter work was generalized in \cite{rN16} to squares of order $8k$ for any $k\in \Z^+$.

\subsection{Most-perfect Squares}
This article makes vital use of most-perfect squares. Let $n$ be a natural number divisible by $p$. A natural pandiagonal magic square $R$ of order $n$ is said to be a {\bf most-perfect square of type-$p$} if the following two properties hold:
  \begin{itemize}
  \item[(i)] ({\bf Complementary property}) Starting from any location in $R$, consider the symbol in that location together with the $p-1$ other symbols lying in the same broken main-diagonal $n/p$ units apart from one another. The sum of these symbols is $\displaystyle \frac{p(n^2-1)}{2}$.
  \item[(ii)] ({\bf $p\times p$ property}) The symbols in any $p\times p$ subsquare formed from consecutive rows and columns (allowing toric wraparound) sum to $\displaystyle \frac{p^2(n^2-1)}{2}$.
  \end{itemize}

 Examples of type-$2$ and type-$3$ most-perfect squares are given in Figure \ref{f:mps8}.

\begin{figure}[h]
{\tiny
$$
\begin{array}{|cccc|cccc|}
\hline
 0 & 31 & 48 & 47 & 56 & 39 & 8 & 23 \\
 59 & 36 & 11 & 20 & 3 & 28 & 51 & 44 \\
 6 & 25 & 54 & 41 & 62 & 33 & 14 & 17 \\
 61 & 34 & 13 & 18 & 5 & 26 & 53 & 42 \\
 \hline
 7 & 24 & 55 & 40 & 63 & 32 & 15 & 16 \\
 60 & 35 & 12 & 19 & 4 & 27 & 52 & 43 \\
 1 & 30 & 49 & 46 & 57 & 38 & 9 & 22 \\
 58 & 37 & 10 & 21 & 2 & 29 & 50 & 45 \\
 \hline
\end{array}
\qquad
\begin{array}{|ccc|ccc|ccc|}
\hline
 0 & 16 & 23 & 63 & 79 & 59 & 45 & 34 & 41 \\
 64 & 80 & 57 & 46 & 35 & 39 & 1 & 17 & 21 \\
 47 & 33 & 40 & 2 & 15 & 22 & 65 & 78 & 58 \\
 \hline
 7 & 14 & 18 & 70 & 77 & 54 & 52 & 32 & 36 \\
 71 & 75 & 55 & 53 & 30 & 37 & 8 & 12 & 19 \\
 51 & 31 & 38 & 6 & 13 & 20 & 69 & 76 & 56 \\
 \hline
 5 & 9 & 25 & 68 & 72 & 61 & 50 & 27 & 43 \\
 66 & 73 & 62 & 48 & 28 & 44 & 3 & 10 & 26 \\
 49 & 29 & 42 & 4 & 11 & 24 & 67 & 74 & 60 \\
 \hline
\end{array}
$$}
    \caption{Left: A type-$2$ (classical) most-perfect square of order-8. Right: A type-$3$ most-perfect square
    of order $9$. The gridlines serve as an aid in locating complementary entries.}
    \label{f:mps8}
\end{figure}

Type-$p$ most-perfect squares specialize to classical most-perfect squares when $p=2$, in which case $n$ must be doubly even \cite{cP19}. The tasks of counting and constructing classical most-perfect squares were first approached by McClintock \cite{eM97} and culminate in the work of Ollerenshaw and Bree \cite{kO98}, which gives a count of the classical most-perfect squares for any doubly even order $n$, along with a construction method for all such squares. As mentioned above, classical most-perfect squares are used in \cite{pL06} and \cite{rN16} for constructing Franklin squares. When $p\geq 2$, a linear construction of type-$p$ most-perfect squares of order $p^r$ ($r\geq 2$) is given in \cite{jL18}.

\subsection{Type-$p$ Franklin Squares}\label{ss:dubya}
Let $p$ be prime. We say that a natural square $S$ of order $n=kp^3$ is a {\bf Franklin square of type $p$} if it has the following properties:
\begin{itemize}
\item ({\bf $p\times p$ property}): This is as described above for type-$p$ most-perfect squares.
\item ({\bf $1/p$-property for both rows and columns}): We say that $S$ possesses the $1/p$ column property if upon splitting a column of $R$ naturally into $p$ parts, the entries in each part add to  $\frac{1}{p}$ times the magic sum, or rather
$\displaystyle \frac{n(n^2-1)}{2p}$. The {\bf $1/p$ row property} is defined similarly.
\item ({\bf Franklin pattern property}):  The numbers in every Franklin pattern in $S$ add to the magic sum $\ds \frac{n(n^2-1)}{2}$. Franklin patterns specialize to bent diagonals in the case $p=2$. A detailed description of these patterns is given in Section \ref{s:bent}.
\end{itemize}

An example of a type-$3$ Franklin square of order $27$ is given below.

 \scalebox{0.5}{\parbox{\linewidth}{$$
\begin{array}{|ccccccccccccccccccccccccccc|}
\hline
 0 & 691 & 401 & 432 & 151 & 509 & 621 & 259 & 212 & 567 & 286 & 239 & 27 & 718 & 347 & 459 & 97 & 536 & 405 & 124 & 563 & 594 & 313 & 185 & 54 & 664 & 374 \\
 457 & 140 & 495 & 646 & 248 & 198 & 25 & 680 & 387 & 52 & 707 & 333 & 484 & 86 & 522 & 592 & 275 & 225 & 619 & 302 & 171 & 79 & 653 & 360 & 430 & 113 & 549 \\
 \hline \hline
 \boxed{635} & 261 & \multicolumn{1}{c|}{196} & 14 & 693 & \multicolumn{1}{c|}{385} & 446 & 153 & \multicolumn{1}{c||}{493} & 473 & 99 & \multicolumn{1}{c|}{520} & 581 & 288 & \multicolumn{1}{c|}{223} & 41 & 720 & \multicolumn{1}{c||}{331} & 68 & 666 & \multicolumn{1}{c|}{358} & 419 & 126 & \multicolumn{1}{c|}{547} & 608 & \boxed{315} & \boxed{169} \\
 16 & \boxed{698} &   \multicolumn{1}{c|}{\boxed{378}} & 448 & 158 & \multicolumn{1}{c|}{486} & 637 & 266 & \multicolumn{1}{c||}{189} & 583 & 293 & \multicolumn{1}{c|}{216} & 43 & 725 & \multicolumn{1}{c|}{324} & 475 & 104 & \multicolumn{1}{c||}{513} & 421 & 131 & \multicolumn{1}{c|}{540} & 610 & 320 & \multicolumn{1}{c|}{162} & \boxed{70} & 671 & 351 \\
 437 & 144 & \multicolumn{1}{c|}{511} & 626 & 252 & \multicolumn{1}{c|}{214} & 5 & 684 & \multicolumn{1}{c||}{403} & 32 & 711 & \multicolumn{1}{c|}{349} & \boxed{464} &\boxed{90} & \multicolumn{1}{c|}{\boxed{538}} & 572 & 279 & \multicolumn{1}{c||}{241} & 599 & 306 & \multicolumn{1}{c|}{187} & 59 & 657 & \multicolumn{1}{c|}{376} & 410 & 117 & 565 \\
 \hline
 639 & 250 & \multicolumn{1}{c|}{203} & \boxed{18} & 682 & \multicolumn{1}{c|}{392} & 450 & 142 & \multicolumn{1}{c||}{500} & 477 & 88 & \multicolumn{1}{c|}{527} & 585 & 277 & \multicolumn{1}{c|}{230} & 45 & 709 & \multicolumn{1}{c||}{338} & 72 & 655 & \multicolumn{1}{c|}{365} & 423 & \boxed{115} & \multicolumn{1}{c|}{\boxed{554}} & 612 & 304 & 176 \\
 23 & 675 & \multicolumn{1}{c|}{394} & 455 & \boxed{135} & \multicolumn{1}{c|}{\boxed{502}} & 644 & 243 & \multicolumn{1}{c||}{205} & 590 & 270 & \multicolumn{1}{c|}{232} & 50 & 702 & \multicolumn{1}{c|}{340} & 482 & 81 & \multicolumn{1}{c||}{529} & 428 & 108 & \multicolumn{1}{c|}{556} & \boxed{617} & 297 & \multicolumn{1}{c|}{178} & 77 & 648 & 367 \\
 441 & 160 & \multicolumn{1}{c|}{491} & 630 & 268 & \multicolumn{1}{c|}{194} & 9 & 700 & \multicolumn{1}{c||}{383} & 36 & \boxed{727} & \multicolumn{1}{c|}{\boxed{329}} & 468 & 106 & \multicolumn{1}{c|}{518} & \boxed{576} & 295 & \multicolumn{1}{c||}{221} & 603 & 322 & \multicolumn{1}{c|}{167} & 63 & 673 & \multicolumn{1}{c|}{356} & 414 & 133 & 545 \\
 \hline
 628 & 257 & \multicolumn{1}{c|}{207} & 7 & 689 & \multicolumn{1}{c|}{396} & \boxed{439} & 149 & \multicolumn{1}{c||}{504} & 466 & 95 & \multicolumn{1}{c|}{531} & 574 & 284 & \multicolumn{1}{c|}{234} & 34 & 716 & \multicolumn{1}{c||}{342} & 61 & \boxed{662} & \multicolumn{1}{c|}{\boxed{369}} & 412 & 122 & \multicolumn{1}{c|}{558} & 601 & 311 & 180 \\
 21 & 676 & \multicolumn{1}{c|}{395} & 453 & 136 & \multicolumn{1}{c|}{503} & 642 & \boxed{244} & \multicolumn{1}{c||}{\boxed{206}} & 588 & 271 & \multicolumn{1}{c|}{233} & 48 & 703 & \multicolumn{1}{c|}{341} & 480 & 82 & \multicolumn{1}{c||}{530} & \boxed{426} & 109 & \multicolumn{1}{c|}{557} & 615 & 298 & \multicolumn{1}{c|}{179} & 75 & 649 & 368 \\
 442 & 161 & \multicolumn{1}{c|}{489} & 631 & 269 & \multicolumn{1}{c|}{192} & 10 & 701 & \multicolumn{1}{c||}{381} & \boxed{37} & 728 & \multicolumn{1}{c|}{327} & 469 & 107 & \multicolumn{1}{c|}{516} & 577 & \boxed{296} & \multicolumn{1}{c||}{\boxed{219}} & 604 & 323 & \multicolumn{1}{c|}{165} & 64 & 674 & \multicolumn{1}{c|}{354} & 415 & 134 & 543 \\
 \hline \hline
 629 & 255 & 208 & 8 & 687 & 397 & 440 & 147 & 505 & 467 & 93 & 532 & 575 & 282 & 235 & 35 & 714 & 343 & 62 & 660 & 370 & 413 & 120 & 559 & 602 & 309 & 181 \\
 1 & 692 & 399 & 433 & 152 & 507 & 622 & 260 & 210 & 568 & 287 & 237 & 28 & 719 & 345 & 460 & 98 & 534 & 406 & 125 & 561 & 595 & 314 & 183 & 55 & 665 & 372 \\
 458 & 138 & 496 & 647 & 246 & 199 & 26 & 678 & 388 & 53 & 705 & 334 & 485 & 84 & 523 & 593 & 273 & 226 & 620 & 300 & 172 & 80 & 651 & 361 & 431 & 111 & 550 \\
 633 & 262 & 197 & 12 & 694 & 386 & 444 & 154 & 494 & 471 & 100 & 521 & 579 & 289 & 224 & 39 & 721 & 332 & 66 & 667 & 359 & 417 & 127 & 548 & 606 & 316 & 170 \\
 17 & 696 & 379 & 449 & 156 & 487 & 638 & 264 & 190 & 584 & 291 & 217 & 44 & 723 & 325 & 476 & 102 & 514 & 422 & 129 & 541 & 611 & 318 & 163 & 71 & 669 & 352 \\
 435 & 145 & 512 & 624 & 253 & 215 & 3 & 685 & 404 & 30 & 712 & 350 & 462 & 91 & 539 & 570 & 280 & 242 & 597 & 307 & 188 & 57 & 658 & 377 & 408 & 118 & 566 \\
 640 & 251 & 201 & 19 & 683 & 390 & 451 & 143 & 498 & 478 & 89 & 525 & 586 & 278 & 228 & 46 & 710 & 336 & 73 & 656 & 363 & 424 & 116 & 552 & 613 & 305 & 174 \\
 15 & 697 & 380 & 447 & \boxed{157} & 488 & 636 & 265 & 191 & 582 & 292 & 218 & 42 & 724 & 326 & 474 & 103 & 515 & 420 & 130 & 542 & 609 & 319 & 164 & 69 & 670 & 353 \\
 436 & 146 & 510 & 625 & \boxed{254} & 213 & 4 & 686 & 402 & 31 & 713 & 348 & \boxed{463} & \boxed{92} & \boxed{537} & 571 & 281 & 240 & 598 & 308 & 186 & 58 & 659 & 375 & 409 & 119 & 564 \\
 641 & 249 & 202 & 20 & \boxed{681} & 391 & 452 & 141 & 499 & 479 & 87 & 526 & \boxed{587} & \boxed{276} & \boxed{229} & 47 & 708 & 337 & 74 & 654 & 364 & 425 & 114 & 553 & 614 & 303 & 175 \\
 22 & 677 & 393 & 454 & \boxed{137} & 501 & 643 & 245 & 204 & 589 & 272 & 231 & \boxed{49} & \boxed{704} & \boxed{339} & 481 & 83 & 528 & 427 & 110 & 555 & 616 & 299 & 177 & 76 & 650 & 366 \\
 \boxed{443} & \boxed{159} & \boxed{490} &\boxed{632} & \boxed{267} & \boxed{193} & \boxed{11} & \boxed{699} & \boxed{382} & 38 & 726 & 328 & 470 & 105 & 517 & 578 & 294 & 220 & 605 & 321 & 166 & 65 & 672 & 355 & 416 & 132 & 544 \\
 627 & 256 & 209 & 6 & \boxed{688} & 398 & 438 & 148 & 506 & 465 & 94 & 533 & 573 & 283 & 236 & 33 & 715 & 344 & 60 & 661 & 371 & 411 & 121 & 560 & 600 & 310 & 182 \\
 2 & 690 & 400 & 434 & \boxed{150} & 508 & 623 & 258 & 211 & 569 & 285 & 238 & 29 & 717 & 346 & 461 & 96 & 535 & 407 & 123 & 562 & 596 & 312 & 184 & 56 & 663 & 373 \\
 456 & 139 & 497 & 645 & \boxed{247} & 200 & 24 & 679 & 389 & 51 & 706 & 335 & 483 & 85 & 524 & 591 & 274 & 227 & 618 & 301 & 173 & 78 & 652 & 362 & 429 & 112 & 551 \\
 634 & 263 & 195 & 13 & \boxed{695} & 384 & 445 & 155 & 492 & 472 & 101 & 519 & 580 & 290 & 222 & 40 & 722 & 330 & 67 & 668 & 357 & 418 & 128 & 546 & 607 & 317 & 168 \\
 \hline
\end{array}
    $$}}

In the lower region of this square the boxed entries indicate the $1/3$-row and column properties and the $3\times 3$ property. In the upper portion of the square we observe a collection of boxed entries sitting within a frame of $3\times 3$ subsquares. These boxed entries, when taken together, look like the letter ``W." This collection of entries is a Franklin-up pattern. These entries add to the magic sum and can be translated vertically throughout the square (with vertical wraparound). There are also analogous downward Franklin patterns, as well as left and right versions. A detailed description of Franklin patterns is given in Section \ref{s:bent}. In Sections \ref{s:cube} and \ref{s:greater} we show that type-$p$ Franklin squares exist in orders $p^r$ with $r\geq 3$.

Inspiration for these results comes chiefly from \cite{pL06} and \cite{rN16}, where the authors introduce an involution $\theta$ that maps classical most-perfect squares into pandiagonal classical Franklin squares. This involution may be generalized (see Section \ref{s:involution}) so that it applies to type-$p$ most-perfect squares, examples of which exist in all orders $p^r$ with $r\geq 2$ by \cite{jL18}. Therefore, in searching for a reasonable definition for type-$p$ Franklin squares, one could do worse than studying
$\theta (R)$ where $R$ is a type-$p$ most-perfect square. One readily finds that $\theta (R)$ is pandiagonal, has the $p\times p$ property, and has the $1/p$-row and column properties (see Section \ref{s:involution}). Determining reasonable  Franklin patterns is considerably harder, but we are guided by the complementary property of $R$ and Lemma \ref{l:moremoresums2} (see Sections \ref{s:cube} and \ref{s:greater}). The type-$3$ order-$27$ Franklin square given above has the form $\theta (R)$, where $R$ is a (linear) most-perfect square constructed using the method of \cite{jL18}.

\section{An Involution and its Application to Most-Perfect Squares of Type-$p$}\label{s:involution}
Let $n=kp^r$ with $r\geq 2$ and let $R$ be an array of order $n$. We may view $R$ as an order-$p^2$ array
\begin{equation}\label{e:arrayR}
R=(R_{i,j}) \quad \text{with}\quad 0\leq i,j\leq p^2-1,
\end{equation}
 where each $R_{i,j}$ is an array of order $\frac{n}{p^2}$. We define an involution $\theta$ on arrays of order $n$ by
\begin{equation}\label{e:theta}
[\theta (R)]_{i,j}=R_{\bar i,\bar j}
\end{equation}
where, if $i=\ell p +m$ with $\ell ,m \in \{0,1,\dots ,p-1\}$ then $\bar i= mp+\ell$.
We emphasize that $\theta$ depends on $p$.

By way of illustration, if $p=2$ then
$$
R={\tiny \begin{array}{|c|c||c|c|} \hline R_{0,0} &R_{0,1} &R_{0,2} & R_{0,3} \\ \hline
R_{1,0} & R_{1,1} & R_{1,2} & R_{1,3} \\ \hline\hline R_{2,0} & R_{2,1} & R_{2,2} & R_{2,3} \\ \hline
R_{3,0}& R_{3,1} & R_{3,2} & R_{3,3}\\ \hline \end{array}}
\quad \Longrightarrow
\theta (R)=
{\tiny \begin{array}{|c|c||c|c|} \hline R_{0,0} &R_{0,2} &R_{0,1} & R_{0,3} \\ \hline
R_{2,0} & R_{2,2} & R_{2,1} & R_{2,3} \\ \hline \hline R_{1,0} & R_{1,2} & R_{1,1} & R_{1,3} \\ \hline
R_{3,0}& R_{3,2} & R_{3,1} & R_{3,3}\\ \hline \end{array}.}
   $$

Likewise, if $p=3$ then

$$
R={\tiny \begin{array}{|ccc||ccc||ccc|}
 \hline
 R_{0,0} &R_{0,1} &R_{0,2} & R_{0,3}& R_{0,4} & R_{0,5} & R_{0,6} & R_{0,7} & R_{0,8} \\
 R_{1,0} &R_{1,1} &R_{1,2} & R_{1,3}& R_{1,4} & R_{1,5} & R_{1,6} & R_{1,7} & R_{1,8} \\
 R_{2,0} &R_{2,1} &R_{2,2} & R_{2,3}& R_{2,4} & R_{2,5} & R_{2,6} & R_{2,7} & R_{2,8} \\
 \hline\hline
 R_{3,0} &R_{3,1} &R_{3,2} & R_{3,3}& R_{3,4} & R_{3,5} & R_{3,6} & R_{3,7} & R_{3,8} \\
 R_{4,0} &R_{4,1} &R_{4,2} & R_{4,3}& R_{4,4} & R_{4,5} & R_{4,6} & R_{4,7} & R_{4,8} \\
 R_{5,0} &R_{5,1} &R_{5,2} & R_{5,3}& R_{5,4} & R_{5,5} & R_{5,6} & R_{5,7} & R_{5,8} \\
 \hline\hline
 R_{6,0} &R_{6,1} &R_{6,2} & R_{6,3}& R_{6,4} & R_{6,5} & R_{6,6} & R_{6,7} & R_{6,8} \\
 R_{7,0} &R_{7,1} &R_{7,2} & R_{7,3}& R_{7,4} & R_{7,5} & R_{7,6} & R_{7,7} & R_{7,8} \\
 R_{8,0} &R_{8,1} &R_{8,2} & R_{8,3}& R_{8,4} & R_{8,5} & R_{8,6} & R_{8,7} & R_{8,8} \\
\hline
 \end{array}}
   $$
implies
$$
\theta (R )=
{\tiny \begin{array}{|ccc||ccc||ccc|}
 \hline
 R_{0,0} &R_{0,3} &R_{0,6} & R_{0,1}& R_{0,4} & R_{0,7} & R_{0,2} & R_{0,5} & R_{0,8} \\
 R_{3,0} &R_{3,3} &R_{3,6} & R_{3,1}& R_{3,4} & R_{3,7} & R_{3,2} & R_{3,5} & R_{3,8} \\
 R_{6,0} &R_{6,3} &R_{6,6} & R_{6,1}& R_{6,4} & R_{6,7} & R_{6,2} & R_{6,5} & R_{6,8} \\
 \hline\hline
 R_{1,0} &R_{1,3} &R_{1,6} & R_{1,1}& R_{1,4} & R_{1,7} & R_{1,2} & R_{1,5} & R_{1,8} \\
 R_{4,0} &R_{4,3} &R_{4,6} & R_{4,1}& R_{4,4} & R_{4,7} & R_{4,2} & R_{4,5} & R_{4,8} \\
 R_{7,0} &R_{7,3} &R_{7,6} & R_{7,1}& R_{7,4} & R_{7,7} & R_{7,2} & R_{7,5} & R_{7,8} \\
 \hline\hline
 R_{2,0} &R_{2,3} &R_{2,6} & R_{2,1}& R_{2,4} & R_{2,7} & R_{2,2} & R_{2,5} & R_{2,8} \\
 R_{5,0} &R_{5,3} &R_{5,6} & R_{5,1}& R_{5,4} & R_{5,7} & R_{5,2} & R_{5,5} & R_{5,8} \\
 R_{8,0} &R_{8,3} &R_{8,6} & R_{8,1}& R_{8,4} & R_{8,7} & R_{8,2} & R_{8,5} & R_{8,8} \\
\hline
 \end{array}.}
 $$
The mapping $\theta$ specializes to the involution given in \cite{pL06} in the case $p=2$ and $n=8$; the reader may check that if $R$ is the square in the left portion of Figure \ref{f:mps8}, then $\theta (R)$ is a Franklin square of order $8$.

It is our intention to provide examples of type-$p$ Franklin squares by applying $\theta$ to most-perfect squares of type $p$. We begin this process over the next several results, culminating
in Proposition \ref{p:culminate}.

\begin{prop}\label{p:realpbyp}
Suppose $n$ is triply divisible by $p$ and  that $R$ is a square of order $n$ possessing the
$p\times p$ property. Then $\theta (R)$ has the $p\times p$ property.
\end{prop}

\begin{proof}
Observe that $R$ has the $p\times p$ property if and only if for any $(p+1)\times (p+1)$-subsquare $A$ of $R$ formed from consecutive rows and columns (allowing wraparound), with
$$
A=\begin{array}{|cccc|c|}
\hline
a_{11} & a_{12} & \dots & a_{1p} & a_{1,p+1}\\
a_{21} & a_{22} & \dots & a_{2p} & a_{2,p+1}\\
\vdots & \vdots & \vdots& \vdots &\vdots \\
a_{p1} & a_{p2} &\dots  & a_{pp} & a_{p,p+1} \\
\hline
a_{p+1,1} & a_{p+1,2} & \dots & a_{p+1,p} & a_{p+1,p+1}\\
\hline
\end{array},
  $$
we have  $\ds \sum_{j=1}^p a_{1j} = \sum_{j=1}^p a_{p+1,j}$ and $\ds \sum_{j=1}^p a_{j1} = \sum_{j=1}^p a_{j,p+1}$. Also, we may define variants $\theta _{row}$ and $\theta _{col}$ of $\theta$ by
$$
[\theta _{row} (R)]_{i,j}=R_{\bar i ,j} \qquad \text{and}\qquad [\theta _{col} (R)]_{i,j}=R_{i,\bar j},
   $$
where $\bar i$ and $\bar j$ are as in (\ref{e:theta}).

We first show that $\theta _{row} (R)$ possesses the $p\times p$ property. We may view obtaining $\theta _{row} (R)$ from $R$ by swapping one pair of rows at a time. Let $r$ be a row of $R$ lying in the band $R_{i,0}, R_{i,1}, \dots , R_{i,p^2-1}$ of subsquares. According to the definition of $\theta _{row}$, we swap $r$ with a row $\bar r$ in $R$ that lies in the same relative position
in the band of subsquares $R_{\bar i,0}, R_{\bar i, 1},\dots ,R_{\bar i, p^2-1}$. Therefore $r$ is being swapped with a row $\bar r$ that lies $|i-\bar i|(n/p^2)$ units distant from $r$. Because $\frac{n}{p^2}$ is a multiple of $p$, the characterization of the $p\times p$ property given at the beginning of this proof indicates that the $p\times p$ property remains intact after this row swap. It follows that $\theta_{row}(R)$ possesses the $p\times p$ property. A similar argument shows that $\theta_{col}(R)$ possesses the $p\times p$ property, and a combination of these two  results gives that $\theta (R)
=\theta _{col}(\theta _{row}(R))$ possesses the $p\times p$ property.
\end{proof}

\begin{prop}\label{p:pcolprop}
Let $n$ be triply divisible by a prime $p$ and let $R$ be a type-$p$ most-perfect square of order $n$. Then $\theta(R)$ has the $1/p$ row and column properties.
\end{prop}

\begin{proof}
It suffices  to  show that $\theta _{row}(R)$ has the $1/p$ column property.
First we establish some notation: Fix $k\in \{0,\dots,\frac{n}{p^2}-1\}$ and let
$\sigma_{i,j}$ denote the sum of the entries in the $k$-th column of $R_{i,j}$. This sum has $n/p^2$ terms, a fact that will be important later in the proof. Similarly let
$\tilde \sigma_{i,j}$ denote the sum of the entries in the $k$-th column of $[\theta_{row}(R)]_{i,j}$. Recall throughout that
$i,j\in \{0,1,\dots, p^2-1\}$.

 Observe that $\tilde \sigma _{0,j}+\cdots + \tilde \sigma_{p-1,j}$ is the sum of the first $n/p$ entries of the
 $j\cdot \frac{n}{p^2} +k$ column of $\theta _{row}(R)$. (We could address another collection of $n/p$ entries in this same column by replacing $\tilde \sigma _{0,j}$ with $\tilde \sigma _{i+0,j}$, etc., but this clutters the indices so we consider the top $n/p$ entries only.) Applying Equation (\ref{e:theta}), the $p\times p$ property of $R$ (actually the characterization given at the beginning of the proof of Proposition \ref{p:realpbyp}), and the complementary property of $R$ in succession, we obtain
 \begin{align*}
 \tilde \sigma _{0,j}+\tilde \sigma _{1,j}+\cdots +\tilde \sigma _{p-1,j}
 &=  \sigma _{0,j}+\sigma _{p,j}+\sigma _{2p,j}+\cdots + \sigma _{(p-1)p,j}\\
 &=  \sigma _{0,j}+\sigma _{p,j+p}+\sigma _{2p,j+2p}+\cdots + \sigma _{(p-1)p,j+(p-1)p}\\
 &= \frac{n}{p^2}\cdot \frac{p(n^2-1)}{2}=\frac{n(n^2-1)}{2p},
 \end{align*}
as desired. The use of the complementary property to obtain the last line of the displayed equation requires a bit more explanation: Gather the first terms of each sum $\sigma _{\ell p,j+\ell p}$. These add to $\frac{p(n^2-1)}{2}$ by the complementary property, as does the collection of second terms, etc. Since each $\sigma _{\ell p,j+\ell p}$ has $\frac{n}{p^2}$ terms, we obtain $\frac{p(n^2-1)}{2}$ exactly $\frac{n}{p^2}$ times.
\end{proof}

Next we go about showing that if $R$ is a type-$p$ most-perfect square of order $n$, then $\theta (R)$ is pandiagonal. We begin with a pair of lemmas.
\begin{lem}\label{l:diagsum}
Let $m,n\in \N$ and consider a nonnegative integer array $A$ of size  $(mp+1)\times (np+1)$ with
$$
A= \begin{array}{c|ccc|c}
a & \ & v & \ & b \\
\hline
\ & \ & \ & \ & \ \\
u & \ & D &\ & w\\
\ & \ & \ & \ & \ \\
\hline
c & \ & z & \ & d
\end{array}.
   $$
 Here $a,b,c,d\in \Z$, $u,w$ are lists of length $mp-1$, $v,z$ are lists of length $np-1$, and $D$ is an $(mp-1) \times (np-1)$ array. If $A$ possesses the $p\times p$ property then $a+ d=c+ b$.
\end{lem}

\begin{proof}
By the $p\times p$ property
$$
a+ u+ v+ D =b+ v+ w + D=c+ u+ z+ D =d+ z+ w+ D,
  $$
where the additions indicate the total sums of symbols in each type of list. It follows that
$$
(a+ u+ v+ D) + (d+ z+ w+ D)=(b+ v+ w + D)+ (c+ u+ z+ D ),
  $$
and cancellation gives the result.
\end{proof}

\begin{lem}\label{l:transversalsum}
Let $m\in \Z^+$ and $A=(a_{i,j})$ be an $m\times m$ array such that if $\begin{pmatrix} a_{i_1,j_1} & a_{i_1,j_2}\\
a_{i_2,j_1} & a_{i_2,j_2}\end{pmatrix}$ is a $2\times 2$ subarray of $A$, then $a_{i_1,j_1}+a_{i_2,j_2}=a_{i_1,j_2}+a_{i_2,j_1}$.  Then all transversals of $A$ have the same sum.
\end{lem}

\begin{proof}
Let
$T=\{a_{1,j_1},a_{2,j_2},\dots ,a_{m,j_m}\}$ be a transversal for $A$. We show that the sum of the elements of $T$ equals the sum of the main diagonal elements of $A$. This is done by constructing a chain of transversals, culminating in the diagonal transversal, each of which has the same sum. We form a new transversal $T_1$ from $T$ as follows: if $a_{1,j_1}=a_{1,1}$, then $T_1=T$. If $a_{1,j_1}\ne a_{1,1}$, then, because $T$ is a transversal, there exists $1<k\leq m$ with $j_k=1$. Using the the fact that $j_k=1$ and the array property in the hypothesis, we have  that
$$
a_{k,j_1}+a_{1,1}=a_{k,j_1}+a_{1,j_k}=a_{1,j_1}+ a_{k,j_k}.
  $$
  So if we declare $T_1$ to be the set we obtain from $T$ by replacing $a_{1,j_1}$ and $a_{k,j_k}$ by $a_{1,1}$ and $a_{k,j_1}$, then $T_1$ and $T$ have the same sum, and, importantly, $a_{1,1}\in T_1$. Furthermore, $T_1$ is a transversal of $A$ because all rows and columns of $A$ are still accounted for in $T_1$.

  Observe  that if we eliminate the first row and column from $A$ and remove $a_{1,1}$ from $T_1$, then the remaining elements of $T_1$ form a transversal of the new array, and we can repeat  the process above to obtain a transversal $T_2=\{a_{1,1},a_{2,2},\dots ,a_{m,j_m}\}$ of $A$ that has the same sum as $T_1$, with $a_{1,1}$ and $a_{2,2}$ in $T_2$. Continuing in this fashion, we see that the sum of  $T$ is equal to the sum of $T_m$, which is the main diagonal transversal of $A$.
\end{proof}

\begin{prop}\label{p:pangeneral}
Let $p$ be prime and $n$ triply divisible by $p$. If $R$ is a type-$p$ most-perfect square of order $n$ then $\theta (R)$ is pandiagonal.
\end{prop}

\begin{proof}
Let $d_0,\dots ,d_{n-1}$ denote the elements of a broken diagonal in $\theta (R)$ with
$d_j$ lying in the $j$-th column of $\theta (R)$. Let $k\in \{0,1,\dots ,\frac{n}{p^2}-1\}$ and put
$$
a_i= d_{i\cdot \frac{n}{p^2}+k} \quad (0\leq i\leq p^2-1).
  $$
We claim that $a_0+a_1+\cdots +a_{p^2-1}=\frac{p^2(n^2-1)}{2}$. If this is true then
$$
\sum_{j=0}^{n-1}d_j=\frac{n}{p^2}\cdot \frac{p^2(n^2-1)}{2} =\frac{n(n^2-1)}{2},
   $$
as desired.

   We set about proving the claim. Due to their construction, all of the $a_k$'s lie in the same (relative) location within an $[\theta (R)]_{i,j}$. Because the mapping $R\mapsto \theta (R)$ is of order two and merely permutes the $R_{i,j}$'s without altering the relative locations of entries  within $R_{i,j}$'s (see Equation (\ref{e:theta})), we also know that if $B=(b_{i,j})$ is the $p^2\times p^2$ subarray of $R$ consisting of {\em all} entries lying in this same relative location within some $R_{i,j}$, then $\{a_0,a_1,\dots ,a_{p^2-1}\}$ is a transversal of $B$. Because $R$ has the $p\times p$ property and $n$ is triply divisible by $p$, we may apply Lemma \ref{l:diagsum} to the various $2\times 2$ subarrays of $B$, and so the hypotheses of Lemma \ref{l:transversalsum} are satisfied for $B$. Therefore $a_0+a_1+\cdots +a_{p^2-1}$ is equal to the sum
   $b_{0,0}+b_{1,1}+\cdots b_{p^2-1,p^2-1}$ of the diagonal transversal of $B$.

   Observe that adjacent terms of the sum $b_{0,0}+b_{1,1}+\cdots b_{p^2-1,p^2-1}$ are actually $\frac{n}{p^2}$ units apart on the main diagonal of $R$. Therefore if we rewrite this sum as
   \begin{align*}
   b_{0,0}+b_{1,1}+\cdots +b_{p^2-1,p^2-1}
   &= (b_{0,0}+b_{p,p}+b_{2p,2p}+\cdots +b_{(p-1)p,(p-1)p})\\
    &+(b_{1,1}+b_{1+p,1+p}+ b_{1+2p,1+2p}+\cdots +b_{1+(p-1)p,1+(p-1)p})\\
   & +\cdots +(b_{p-1,p-1}+b_{2p-1,2p-1}+\cdots +b_{p^2 -1,p^2-1})
     \end{align*}
     then within each parenthetical summand there are $p$ terms and adjacent terms are $n/p$ units apart in $R$. Because $R$ possesses the complementary property, we then know that each parenthetical summand adds to $\frac{p(n^2-1)}{2}$. Because there are $p$ parenthetical summands, we may then conclude that
     $$
     a_0+a_1+\cdots +a_{p^2-1}=b_{0,0}+b_{1,1}+\cdots b_{p^2-1,p^2-1}=p\cdot \frac{p(n^2-1)}{2}=\frac{p^2(n^2-1)}{2}.
     $$
   Therefore the claim is proved.
\end{proof}

We may summarize the previous results as follows:

\begin{prop}\label{p:culminate}
Let $n$ be triply divisible by $p$ and suppose $R$ is a type-$p$ most-perfect square of order $n$. Then $\theta (R)$ is semi-magic, possesses the $p\times p$ property, possesses the $1/p$ row and column properties, and is pandiagonal.
\end{prop}

\section{Defining Type-$p$ Franklin Squares: Bent Diagonals}\label{s:bent}
In the introduction we established precise characteristics of type-$p$ Franklin squares, with the exception of the bent diagonals, which we address presently. We will refer to the type-$p$ analogs of bent diagonals as {\bf Franklin patterns}.
In the interest of simplicity we describe Franklin patterns first in the special case $n=p^3$ before addressing the general case $n=kp^3$ (Section \ref{s:greater}). These squares, except for the smallest few primes, are large, so we will be using the special cases $p=2,3,5$ to illustrate several key points. Also, we will first focus our attention on the construction of a particular Franklin pattern, called a {\bf Franklin-up pattern}, an example of which is given in Section \ref{ss:dubya}. These patterns specialize to classical Franklin ``V" patterns when $p=2$.

Consider a collection of $n/p =p^2$ consecutive rows of $S$, which we intend to serve as a {\bf frame} for a Franklin-up pattern $W$. This frame can be partitioned into a $p\times p^2$ array $T$ whose entries are subsquares $T_{i,j}$, each of size $p\times p$, where $0\leq i\leq p$ and $0\leq j\leq p^2-1$. Square $T_{i,j}$, which we occasionally refer to as a {\bf block}, lies in the $i$-row and $j$-column of $T$. We describe which subsquares of $T$ have non-trivial intersection with $W$. The array $T$ can be partitioned into $p\times p$ subarrays $B_0,\dots ,B_{p-1}$ (called {\bf bands}), each containing $p$ columns of $T$, where $B_0$ contains the leftmost $p$ columns of $T$, $B_1$ contains the next $p$ columns of $T$, and so on. For $0\leq j < \frac{p-1}{2}$, the Franklin-up pattern $W$ intersects each entry of the main diagonal of $B_j$ when $j$ is even, and each entry of the off-diagonal of $B_j$ when $j$ is odd. The locations of these  intersections reflect across the central band $B_{(p-1)/2}$, so that $W$ intersects each entry of the off-diagonal of $B_{(p-1)-j}$ when $j$ is even, and each entry of the main diagonal of $B_{(p-1)-j}$ when $j$ is odd. When $p$ is odd there will be a central band $B_{(p-1)/2}$, in which intersection with $W$ will rise to a central peak when $(p-1)/2$ is odd and fall to central valley when $(p-1)/2$ is even. These intersections of $W$ with $T$ are indicated below in cases $p=2,3,5$; double vertical lines separate bands.
$$
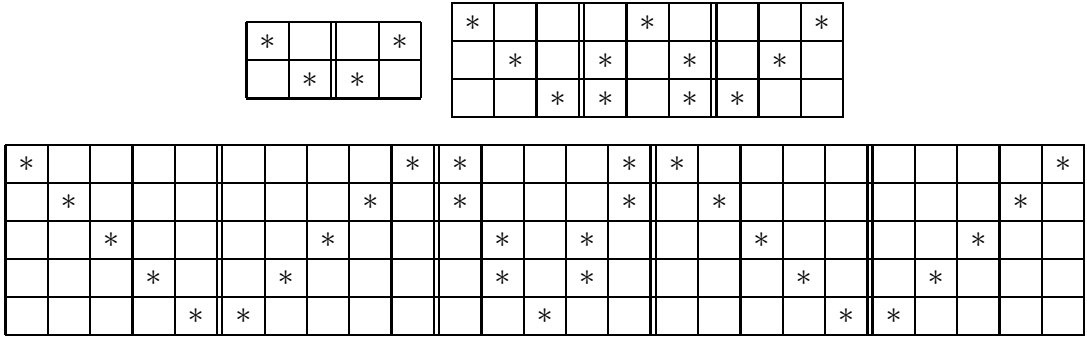
\begin{array}{|c|c||c|c|}
\hline
\ast & \ & \ & \ast \\
\hline
\ & \ast  & \ast & \ \\
\hline
\end{array}
\quad
\begin{array}{|c|c|c||c|c|c||c|c|c|}
\hline
\ast & \ & \ & \ & \ast & \ & \ & \ & \ast  \\
\hline
\ & \ast & \ & \ast & \ & \ast & \ & \ast & \  \\
\hline
\ & \ & \ast & \ast & \ & \ast  & \ast & \ & \  \\
\hline
\end{array}
   $$
$$
\begin{array}{|c|c|c|c|c||c|c|c|c|c||c|c|c|c|c||c|c|c|c|c||c|c|c|c|c|}
\hline
\ast & \ & \ & \ & \ & \ & \ & \ & \ & \ast & \ast & \ & \ & \ & \ast & \ast & \ & \ & \ & \ & \ & \ & \ & \ & \ast   \\
\hline
\ & \ast & \ & \ & \ & \ & \ & \ & \ast & \ & \ast & \ & \ & \ & \ast & \ & \ast & \ & \ & \ & \ & \ & \ & \ast & \   \\
\hline
\ & \ & \ast & \ & \ & \ & \ & \ast & \ & \ & \ & \ast & \ & \ast & \ & \ & \ & \ast & \ & \ & \ & \ & \ast & \ & \   \\
\hline
\ & \ & \ & \ast & \ & \ & \ast & \ & \ & \ & \ & \ast & \ & \ast & \ & \ & \ & \ & \ast & \ & \ & \ast & \ & \ & \   \\
\hline
\ & \ & \ & \ & \ast & \ast & \ & \ & \ & \ & \ & \ & \ast & \ & \ & \ & \ & \ & \ & \ast & \ast & \ & \ & \ & \   \\
\hline
\end{array}
    $$
In the figure above, we emphasize that each small rectangle represents some $p\times p$ array $T_{i,j}$ in $T$, not an individual entry in $S$.

In case the description above is not sufficiently specific, the Franklin-up pattern we construct in this frame will intersect the following subsquares:
$$
T_{j,2mp+j} \text{ and } T_{j,(p^2-1)-(2mp+j)} \text{ for } 0\leq j\leq p-1 \text{ and } 0\leq m< \frac{p-1}{4},
  $$
  and
$$
T_{j,(2mp-1)-j} \text{ and } T_{j,(p^2-1)-((2mp-1)-j)} \text{ for } 0\leq j\leq p-1 \text{ and } 0< m\leq \frac{p-1}{4}.
    $$
Further, if $p$ is odd, then $W$ will also intersect the following subsquares, depending on the parity of $(p-1)/2$: If $(p-1)/2$ is even then $W$ intersects $T_{p-1,\frac{p^2-1}{2}}$ and
$$
T_{j,\frac{p(p-1)}{2}+\lfloor \frac{j}{2}\rfloor} \text{ and } T_{j,\frac{p(p-1)}{2}+(p-1)-\lfloor \frac{j}{2}\rfloor}
\text{ for } 0\leq j<p-1.
  $$
On the other hand, if $(p-1)/2$  is odd then $W$ intersects $T_{0,\frac{p^2-1}{2}}$ and
$$
T_{j,\frac{p^2-1}{2}-\lceil \frac{j}{2}\rceil} \text{ and } T_{j,\frac{p^2-1}{2}-\lceil\frac{j}{2}\rceil} \text{ for } 0< j \leq p-1.
   $$

We've seen which of the arrays $T_{i,j}$ intersect $W$ non-trivially, and we now need to determine those intersections precisely. For $0\leq j\leq p-1$ with $j\ne (p-1)/2$, we let $B^i_j$ denote the $p\times p$ square in the $i$-th row of $B_j$ that intersects $W$. Further, when $p$ is odd, we let $B^{i,0}_{\frac{p-1}{2}}$ and $B^{i,1}_{\frac{p-1}{2}}$ denote the left and right squares, respectively, in the $i$-th row of $B_{\frac{p-1}{2}}$. These squares will coincide exactly when $i=0$ and $\frac{p-1}{2}$ is odd or when $i=p-1$ and $\frac{p-1}{2}$ is even. (Each $B^i_j$ is a $T_{k,\ell}$ for some $k,\ell$, and while we can make this connection explicitly, it seems unnecessary and perhaps counterproductive.) Below we indicate the positions of the $B^i_j$ in cases $p=2,3,5$:

$$
\begin{array}{|c|c||c|c|}
\hline
B^0_0 & \ & \ & B^0_1 \\
\hline
\ & B^1_0  & B^1_1 & \ \\
\hline
\end{array}
\quad
\begin{array}{|c|c|c||c|c|c||c|c|c|}
\hline
B^0_0 & \ & \ & \ & B^{0,0}_1=B^{0,1}_1 & \ & \ & \ & B^0_2  \\
\hline
\ & B^1_0 & \ & B^{1,0}_1 & \ & B^{1,1}_1 & \ & B^1_2 & \  \\
\hline
\ & \ & B^2_0 & B^{2,0}_1 & \ & B^{2,1}_1  & B^2_2 & \ & \  \\
\hline
\end{array}
   $$
 \scalebox{0.7}{\parbox{\linewidth}{$$
\begin{array}{|c|c|c|c|c||c|c|c|c|c||c|c|c|c|c||c|c|c|c|c||c|c|c|c|c|}
\hline
B^0_0 & \ & \ & \ & \ & \ & \ & \ & \ & B^0_1 & B^{0,0}_2 & \ & \ & \ & B^{0,1}_2 & B^0_3 & \ & \ & \ & \ & \ & \ & \ & \ & B^0_4  \\
\hline
\ & B^1_0 & \ & \ & \ & \ & \ & \ & B^1_1 & \ & B^{1,0}_2 & \ & \ & \ & B^{1,1}_2 & \ & B^1_3 & \ & \ & \ & \ & \ & \ & B^1_4 & \   \\
\hline
\ & \ & B^2_0 & \ & \ & \ & \ & B^2_1 & \ & \ & \ & B^{2,0}_2 & \ & B^{2,1}_2 & \ & \ & \ & B^2_3 & \ & \ & \ & \ & B^2_4 & \ & \   \\
\hline
\ & \ & \ & B^3_0 & \ & \ & B^3_1 & \ & \ & \ & \ & B^{3,0}_2 & \ & B^{3,1}_2 & \ & \ & \ & \ & B^3_3 & \ & \ & B^3_4 & \ & \ & \   \\
\hline
\ & \ & \ & \ & B^4_0 & B^4_1 & \ & \ & \ & \ & \ & \ & B^{4,0}_2=B^{4,1}_2 & \ & \ & \ & \ & \ & \ & B^4_3 & B^4_4 & \ & \ & \ & \   \\
\hline
\end{array}
    $$
    }}

Let $1\leq \alpha , \beta <p$ with $\alpha +\beta =p$, and let $0\leq j< (p-1)/2$. Recall that each $B^i_j$ is a $p\times p$ array. The Franklin-up pattern $W$ will intersect the $B^i_j$ as follows, where in each instance $1\leq i\leq p$.
  \begin{itemize}
  \item If $j$ is even then $B^i_j \cap W$ consists of the first $\alpha$ entries in row $2j$ and the last $\beta$ entries in row $2j+1$ of $B^i_j$.
  \item If $j$ is even then $B^i_{p-1-j}\cap W$ consists of the last $\beta$ entries in row $2j$ and the first $\alpha$ entries in row $2j+1$ of $B^i_{p-1-j}$.
  \item If $j$ is odd then $B^i_j\cap W$ consists of the last $\beta$ entries in row $2j$ and the first $\alpha$ entries in row $2j+1$ of $B^i_j$.
  \item If $j$ is odd then $B^i_{p-1-j}\cap W$ consists of the first $\alpha$ entries of row $2j$ and the last $\beta$ entries of row $2j+1$.
  \end{itemize}
Below is a pictorial representation of these intersections:
\begin{figure}[h]
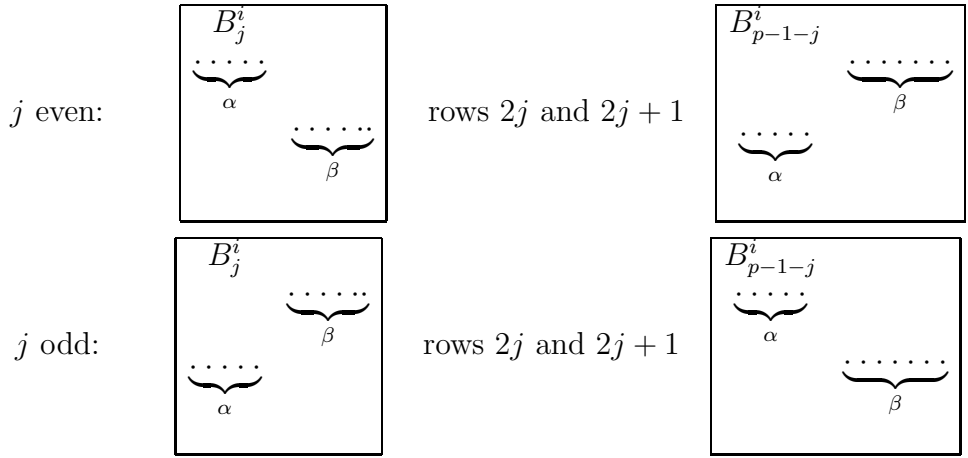

$$
j\text{ even: }\qquad
\begin{array}{|cc|}
 \hline
B_j^i & \ \\
 \underbrace{\cdot \cdot \cdot \cdot \cdot}_{\alpha}  & \ \\
 \ &\underbrace{ \cdot\cdot \cdot \cdot \cdot \cdot }_{\beta} \\
 \ & \ \\
 \hline
 \end{array}
 \quad\text{ rows }2j\text{ and } 2j+1\quad
  \begin{array}{|cc|}
 \hline
 B_{p-1-j}^i & \ \\
 \ & \underbrace{\cdot \cdot \cdot \cdot \cdot \cdot \cdot}_{\beta}   \\
 \underbrace{\cdot \cdot \cdot \cdot \cdot }_{\alpha}  & \ \\
 \ & \ \\
 \hline
 \end{array}
   $$
 $$
j\text{ odd: }\qquad
 \begin{array}{|cc|}
 \hline
 B_j^i & \ \\
  \ &\underbrace{ \cdot\cdot \cdot \cdot \cdot \cdot }_{\beta} \\
 \underbrace{\cdot \cdot \cdot \cdot \cdot}_{\alpha}  & \ \\
 \ & \ \\
 \hline
 \end{array}
 \quad\text{ rows }2j\text{ and } 2j+1\quad
  \begin{array}{|cc|}
 \hline
 B_{p-1-j}^i & \ \\
  \underbrace{\cdot \cdot \cdot \cdot \cdot }_{\alpha}  & \ \\
 \ & \underbrace{\cdot \cdot \cdot \cdot \cdot \cdot \cdot}_{\beta}   \\
 \ & \ \\
 \hline
 \end{array}
   $$
\caption{Intersections of $B_j^i$ and $B_{p-1-j}^i$ with $W$ when $0\leq j <(p-1)/2$.}
    \label{f:blockint}
\end{figure}

It remains to see how, when $p$ is odd, the squares $B^{i,k}_{\frac{p-1}{2}}$ in the central band will intersect $W$:
  \begin{itemize}
  \item If $i$ is even then $B_{\frac{p-1}{2}}^{i,0} \cap W$ consists of the first $\alpha$ entries in the bottom row of $B_{\frac{p-1}{2}}^{i,0}$.
  \item  If $i$ is even then $B_{\frac{p-1}{2}}^{i,1} \cap W$ consists of the last $\beta$ entries in the bottom row of $B_{\frac{p-1}{2}}^{i,1}$.
  \item If $i$ is odd then $B_{\frac{p-1}{2}}^{i,0} \cap W$ consists of the last $\beta$ entries in the bottom row of $B_{\frac{p-1}{2}}^{i,0}$.
  \item  If $i$ is odd then $B_{\frac{p-1}{2}}^{i,1} \cap W$ consists of the first $\alpha$ entries in the bottom row of $B_{\frac{p-1}{2}}^{i,1}$.
  \item In the special case that $B_{\frac{p-1}{2}}^{i,0}=B_{\frac{p-1}{2}}^{i,1}$, their intersection with $W$ consists of the entire bottom row of $B_{\frac{p-1}{2}}^{i,0}$.
  \end{itemize}
Below is a pictorial representation of these intersections:
\begin{figure}[h]
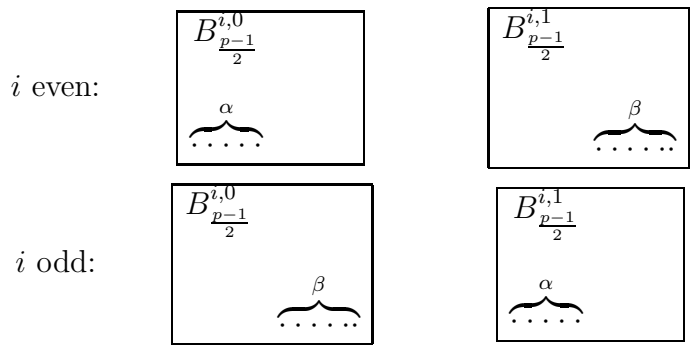

$$
i\text{ even: }\qquad
\begin{array}{|cc|}
 \hline
B_{\frac{p-1}{2}}^{i,0} & \ \\
\   & \  \\
 \overbrace{\cdot \cdot \cdot \cdot \cdot}^{\alpha}  & \qquad \\
 \hline
 \end{array}
 \qquad\qquad
\begin{array}{|cc|}
 \hline
B_{\frac{p-1}{2}}^{i,1} & \ \\
\   & \  \\
 \qquad  &  \overbrace{\cdot \cdot \cdot \cdot \cdot \cdot}^{\beta} \\
 \hline
 \end{array}
   $$
 $$
i\text{ odd: }\qquad
\begin{array}{|cc|}
 \hline
B_{\frac{p-1}{2}}^{i,0} & \ \\
\   & \  \\
 \qquad  &  \overbrace{\cdot \cdot \cdot \cdot \cdot \cdot}^{\beta} \\
 \hline
 \end{array}
\qquad \qquad
\begin{array}{|cc|}
 \hline
B_{\frac{p-1}{2}}^{i,1} & \ \\
\   & \  \\
 \overbrace{\cdot \cdot \cdot \cdot \cdot}^{\alpha}  & \qquad \\
 \hline
 \end{array}
   $$
 \caption{Intersections of $B_{\frac{p-1}{2}}^{i,k}$ with $W$.}
    \label{f:blockint2}
 \end{figure}

 In the case $p=3$, $\alpha=1$, and $\beta=2$, the intersections described above, which characterize a Franklin-up $W$ pattern, are illustrated in the order-$27$ square shown in Section \ref{ss:dubya}.

 We observe that within its frame, a Franklin-up pattern $W$ intersects each column of $S$ exactly once, and each row exactly $p$ times, so $W$ has $n=p^3$ entries. Also, while $W$ does not have vertical midline symmetry when $p>2$, the blocks containing $W$ do possess this symmetry. Finally, we can obtain Franklin-right, Franklin-down, and Franklin-left patterns from a Franklin-up pattern via clockwise rotations of the ambient square $S$ through $90^\circ$, $180^\circ$, and $270^\circ$, respectively. These constitute the entirety of Franklin patterns in $S$, and they specialize to the classical Franklin ``V" patterns when $p=2$. Therefore, we are now able to make the following definition:

\begin{defn}\label{d:Franklin}
We say that a natural square $S$ of order $n=p^3$ is a {\bf Franklin square of type $p$} if it has the $p\times p$ property, the $1/p$-property for both rows and columns, and the numbers in every Franklin pattern in $S$ add to the magic sum $\ds \frac{n(n^2-1)}{2}$.
\end{defn}

The Franklin pattern requirement in Definition \ref{d:Franklin} applies to patterns arising from any partition $\alpha +\beta =p$ with $1\leq \alpha ,\beta <p$. One might reasonably weaken Definition \ref{d:Franklin} by only requiring the existence of a partition $\alpha +\beta$ of $p$ such that all corresponding Franklin patterns have entries adding to the magic sum. Definition \ref{d:Franklin} and its weakened version both specialize to the definition of classical Franklin squares in the case $p=2$.

\section{Construction of Type-$p$ Franklin Squares}\label{s:cube}

Let $p$ be prime  and let $R$ be a type-$p$ most-perfect square of order $p^3$. Such squares exist; a linear construction is given in \cite{jL18}. In this section we show that $S=\theta (R)$ is a pandiagonal type-$p$ Franklin square, where $\theta$ is the involution introduced in Section \ref{s:involution}. Proposition \ref{p:culminate} says $S$ is pandiagonal, has the $1/p$ row and column properties, and has the $p\times p$ property. It remains to show that the Franklin patterns of $S$ (defined in Section \ref{s:bent}) add to the magic sum. A similar verification for orders $p^r$ with $r\geq 3$ is indicated in Section \ref{s:greater}.

\begin{lem}\label{l:moremoresums}
Let $m,n,p\in \N$ with $p\geq 2$, and consider a nonnegative integer array $A$ of size  $(mp+1)\times np$ with
$$
A= \begin{array}{c|ccc|c|c|c}
a & \ & \ & \ & b_1 & \cdots & b_{p-1}\\
\hline
\ & \ & \ & \ & \ & \cdots & \ \\
\ & \ & D &\ & \ & \cdots & \ \\
\ & \ & \ & \ & \ & \cdots & \ \\
\hline
c & \ & \ & \ & d_1 &\cdots & d_{p-1}
\end{array}.
   $$
 Here $a,b_i,c,d_i\in \Z$ for $1\leq i\leq p-1$ and $D$ is an $(mp-1) \times (n-1)p$ array. If $A$ possesses the $p\times p$ property then $\displaystyle a+ \sum_{i=1}^{p-1}b_i=c+ \sum_{i=1}^{p-1}d_i$.
\end{lem}

\begin{proof}
If $n=1$ this follows immediately from the $p\times p$ property, so we assume $n\geq 2$. Rewrite $A$ as
$$
A= \begin{array}{c|ccc|c|c|c|c}
a & \ & \ & \ & b_0 & b_1 & \cdots & b_{p-1}\\
\hline
\ & \ & \ & \ & \ & \ & \cdots & \ \\
\ & \ & D' &\ &\ & \ & \cdots & \ \\
\ & \ & \ & \ &\ & \ & \cdots & \ \\
\hline
c & \ & \ & \ & d_0 &d_1 &\cdots & d_{p-1}
\end{array}
   $$
where $b_0,d_0\in \Z$ and $D'$ is an array of size $(mp-1)\times ((n-1)p-1)$. By Lemma \ref{l:diagsum} we have $a+d_0=c+b_0$. Also, because $A$ has the $p\times p$ property, we
have $b_0+\cdots +b_{p-1}=d_0+\cdots +d_{p-1}$. Therefore
$$
a+ d_0=c+ b_0 \Longrightarrow a+ (\sum_{i=0}^{p-1}b_i-\sum_{i=1}^{p-1}d_i)=c+ b_0\Longrightarrow a+ \sum_{i=1}^{p-1}b_i=c+ \sum_{i=1}^{p-1}d_i.
  $$
\end{proof}

If $A$ as in the lemma has the $p\times p$ property, then the result of the lemma will continue to hold true if all other instances of $p$ are replaced  by a fixed multiple of $p$. Lemma \ref{l:moremoresums} has a useful generalization:

\begin{lem}\label{l:moremoresums2}
Let $m,n,k,p\in \N$ with $p\geq 2$ and $1\leq k < p$, and consider a nonnegative integer array $A$ of size  $(mp+1)\times np$ with
$$
A= \begin{array}{ccc|ccc|ccc}
a_1& \cdots & a_k & \ & \ & \ & b_{k+1} & \cdots & b_p\\
\hline
\ &\ &\ & \ & \ & \ & \ & \cdots & \ \\
\ &\ &\ & \ & D &\ & \ & \cdots & \ \\
\ &\ &\ & \ & \ & \ & \ & \cdots & \ \\
\hline
c_1 &\cdots & c_k & \ & \ & \ & d_{k+1} &\cdots & d_p
\end{array}.
   $$
 Here all entries are integers and $D$ is an $(mp-1) \times (n-1)p$ array. If $A$ possesses the $p\times p$ property then $\displaystyle \sum_{i=1}^k a_i +\sum_{j=1}^{p-k}b_{k+j}=\sum_{i=1}^k c_i+ \sum_{j=1}^{p-k}d_{k+j}$.
\end{lem}

\begin{proof}
If $n=1$ this follows immediately from the $p\times p$ property, so we assume $n\geq 2$. Let $b_1,\dots ,b_k$ be the entries in $A$ immediately preceding $b_{k+1}$ in the same row and $b_{p+1},\dots ,b_{p+k}$ the entries immediately succeeding $b_{p}$ in the same row. Similarly define $d_1,\dots ,d_k$ and $d_{p+1},\dots, d_{p+k}$. Applying Lemma \ref{l:moremoresums} we have
$$
a_j+ (b_{j+1}+\cdots +b_{j+p-1})=c_j + (d_{j+1}+\cdots +d_{j+p-1})
  $$
for $1\leq j\leq k$. Adding gives
$$
\sum_{j=1}^k [a_j+ (b_{j+1}+\cdots +b_{j+p-1})]=\sum_{j=1}^k [c_j + (d_{j+1}+\cdots +d_{j+p-1})].
  $$
Upon rearrangement,  one can see that a great deal of cancellation occurs in the previous equation. Note that by borrowing terms from the first summand and distributing them among the other summands, we obtain
\begin{align*}
\sum_{j=1}^k [a_j+ (b_{j+1}+\cdots +b_{j+p-1})]&=[a _1+(b_{k+1}+\cdots +b_p)]+\sum_{j=2}^k[a_j+(b_{j}+\cdots +b_{j+p-1})]\\
&=\sum_{i=1}^k a_i + \sum_{j=1}^{p-k}b_{k+j}+\sum_{j=2}^{k}(b_{j}+\cdots +b_{j+p-1}).
   \end{align*}
Likewise
$$
\sum_{j=1}^k [c_j + (d_{j+1}+\cdots +d_{j+p-1})]=\sum_{i=1}^k c_i+ \sum_{j=1}^{p-k}d_{k+j}+\sum_{j=2}^{k}(d_{j}+\cdots +d_{j+p-1}).
   $$
Finally, due to  the $p\times p$ property, the sums $\displaystyle \sum_{j=2}^{k}(b_{j}+\cdots +b_{j+p-1})$  and $\displaystyle \sum_{j=2}^{k}(d_{j}+\cdots +d_{j+p-1})$ are equal (in fact they are equal term by term), so cancellation gives
$$
\sum_{i=1}^k a_i + \sum_{j=1}^{p-k}b_{k+j}=\sum_{i=1}^k c_i + \sum_{j=1}^{p-k}d_{k+j},
   $$
as desired.
\end{proof}

Observe that the result of Lemma \ref{l:moremoresums2} still holds if the statement $1\leq k\leq p$ is replaced by
$1\leq k\leq \ell p$ where $\ell \in \Z ^+$.

\begin{thm}\label{t:k1}
Let $p$ be prime and $n=p^3$. If $R$ is a type-$p$ most-perfect square of order $n$, then $\theta (R)$ is an order-$n$ pandiagonal Franklin square of type $p$. Further, such squares $R$ exist for every prime $p$.
\end{thm}

\begin{proof}
Type-$p$ most-perfect squares of order $n=p^3$ exist due to \cite{jL18}. Also, the square $\theta (R)$ has the $1/p$-property for rows and columns, is pandiagonal, and has the $p\times p$ property by Proposition \ref{p:culminate}. It remains to show  that Franklin patterns in $\theta (R)$ add to the magic sum.

    Let $p=\alpha +\beta$ with $1\leq \alpha ,\beta <p$ and let $W$ be a Franklin-up pattern in $\theta (R)$ corresponding to this partition of $p$. We establish the following notation concerning $W$:
    \begin{itemize}
    \item Let $W_j^i$ denote $W\cap B_j^i$ and $w_j^i$ denote the sum of the elements of $W_j^i$ for $0\leq i, j \leq p-1$, with $j\ne \frac{p-1}{2}$.
    \item $W$ intersects $B_j^i$ in two consecutive rows of $B_j^i$. For $0\leq i, j\leq p-1$ with $j\ne \frac{p-1}{2}$, let $W_{j,t}^i$ denote the portion of $W_j^i$ coming from the top-most of these two rows in $B_j^i$, and let $W_{j,b}^i$ denote the portion of $W_j^i$ coming from the bottom-most of these two rows in $B_j^i$. Let $w_{j,t}^i$ denote the sum of the entries in $W_{j,t}^i$ and $w_{j,b}^i$ denote the sum of the entries in $W_{j,b}^i$. Note $W_j^i=W_{j,t}^i\cup W_{j,b}^i$ and $w_j^i=w_{j,t}^i+w_{j,b}^i$. The need for this distinction between ``$t$" and ``$b$" will be made clear later in the proof when we apply Lemma \ref{l:moremoresums2}.
    \item If $p$ is odd, let $W_{\frac{p-1}{2}}^{i,k}$ denote $B_{\frac{p-1}{2}}^{i,k} \cap W$, and let
    $w_{\frac{p-1}{2}}^{i,k}$ denote the sum of the elements of $W_{\frac{p-1}{2}}^{i,k}$.
    \item For $0\leq j < \frac{p-1}{2}$ we put $s_j=\displaystyle \sum_{i=0}^{p-1} (w_j^i +w_{p-1-j}^i)$.
    \item If $p$ is odd, put $s_{\frac{p-1}{2}}=\displaystyle \sum_{i=0}^{p-1} \left( w_{\frac{p-1}{2}}^{i,0}+w_{\frac{p-1}{2}}^{i,1}\right)$. In the special case that $B_{\frac{p-1}{2}}^{i,0}=B_{\frac{p-1}{2}}^{i,1}$, the corresponding term in $s_{\frac{p-1}{2}}$ is just
        $w_{\frac{p-1}{2}}^{i,0}$, not $ w_{\frac{p-1}{2}}^{i,0}+w_{\frac{p-1}{2}}^{i,1}$, as otherwise we would incur duplication.
    \end{itemize}
Observe that the sum of the entries in $W$ is $\displaystyle \sum_{0\leq j\leq \frac{p-1}{2}} s_j$. We claim that
$s_j=p^2(p^6-1)$ when $0\leq j < \frac{p-1}{2}$, and that $s_{\frac{p-1}{2}}=\displaystyle \frac{p^2(p^6-1)}{2}$ when $p$ is odd. Assuming this claim, we have that the sum of the entries of $W$ is
$$
\sum_{0\leq j\leq \frac{p-1}{2}} s_j =\frac{p-1}{2}[p^2(p^6-1)]+\frac{p^2(p^6-1)}{2}=\frac{p^3(p^6-1)}{2}=\frac{n(n^2-1)}{2}
  $$
when $p$ is odd, and the sum is
$$
\sum_{0\leq j\leq \frac{p-1}{2}} s_j =s_0 =p^2(p^6-1)=\frac{p^3(p^6-1)}{2}=\frac{n(n^2-1)}{2}
   $$
when $p=2$. In either case, the sum of the entries of $W$ is the magic sum, as desired.

To finish, we need to verify the claims about the sums $s_j$. We first present an overview: If $S=\theta (R)$, then we can follow the entries in $W\subseteq S$, and hence the terms of the sums $s_j$, back to $R$ by considering $\theta (S )$. Then we use the complementary property of $R$ together with Lemma \ref{l:moremoresums2} to replace sums $s_j$ with equivalent sums $\tilde s_j$ that have the claimed values.

And now on to  details of the argument, which takes two cases: $0\leq j<\frac{p-1}{2}$ and $j=\frac{p-1}{2}$. First suppose that $0\leq j < \frac{p-1}{2}$. Observe that for $0\leq i < p-1$,  each entry of $W_{j,t}^i\cup W_{p-1-j,t}^i$ is $p$ columns distant from its counterpart in $W_{j,t}^{i+1}\cup W_{p-1-j,t}^{i+1}$  in $S=\theta (R)$, with no repetition of columns. (Here ``counterparts" lies in the same relative position within a block.) Further, we note that the columns of the subsquare frame array $T$ for $W$ coincide with the columns of the subsquare array $(S_{\ell ,m})$ as in Equation \ref{e:arrayR}. (This is not  generally true for rows of $T$.) Also, for $0\leq i\leq p-1$, $W_{j,t}^i$ lies wholly within band $B_j$, which in turn coincides with a natural band of $p$ consecutive columns in the subsquare array $(S_{\ell ,m})$. A similar statement is true for $W_{p-1-j,t}^i$. Therefore subsquares in $S_{\ell ,m}$ containing a pair of counterparts in $W_{j,t}^i\cup W_{p-1-j,t}^i$ and $W_{j,t}^{i+1}\cup W_{p-1-j,t}^{i+1}$ must lie in consecutive columns in $S_{\ell ,m}$.
 Taking all of this into account, upon applying Equation (\ref{e:theta}), we find that elements in
$W^i_{j,t}\cup W^i_{p-1-j,t}$ are $p^2=n/p$ columns distant from their counterparts in $W^{i+1}_{j,t}\cup W^{i+1}_{p-1-j,t}$ within $R=\theta (S)$, with no repetition of columns. (Another way to view this is that the squares containing these counterparts are $p$ columns distant in the subsquare array $R_{\ell ,m}$.) These same observations and conclusion are also  true if $W_{j,t}^i \cup W_{p-1-j,t}^i$ is replaced with $W_{j,b}^i \cup W_{p-1-j,b}^i$.

We have established that as $i$ varies from $0$ to $p-1$, elements in $W_{j,t}^i\cup W_{p-j-1,t}^i$  are $p^2=n/p$ columns apart from their counterparts in $W_{j,t}^{i+1}\cup W_{p-j-1,t}^{i+1}$ in $R$, and similarly when ``$t$" is replaced by ``$b$". If these same statements were also true with ``rows" in place of ``columns", then we could repeatedly apply the complementary property of $R$ to obtain
\begin{equation*}
\begin{split}
s_j&=\sum_{i=0}^{p-1} w_j^i+w_{p-1-j}^i \\
&=\sum_{i=0}^{p-1} (w_{j,t}^i +w_{p-1-j,t}^i) +\sum_{i=0}^{p-1} (w_{j,b}^i+w_{p-1-j,b}^i) \\
&=p\left[\frac{p(p^6-1)}{2}\right] +p\left[\frac{p(p^6-1)}{2}\right]
=p^2(p^6-1),
\end{split}
   \end{equation*}
as claimed. (Here the multiplications by $p$ in the penultimate line are due to the fact that there are
$a+b=p$ members of $W_{j,t}^i\cup W_{p-1-j,t}^i$, and similarly for $W_{j,b}^i\cup W_{p-1-j,b}^i$). Unfortunately, because the rows of the frame array $T=(T_{\ell,m})$ do NOT generally coincide with a natural band of $p$ consecutive rows in $(S_{\ell ,m})$, it is not always true that elements in $W_{j,t}^i\cup W_{p-j-1,t}^i$  are $p^2=n/p$ rows apart from their counterparts in $W_{j,t}^{i+1}\cup W_{p-j-1,t}^{i+1}$ in $R$.

Lemma \ref{l:moremoresums2} can be used to rectify this problem. Elements in $W_{j,t}^{i+1}\cup W_{p-j-1,t}^{i+1}$ may not be $p^2=n/p$ rows distant in $R$ from elements in $W_{j,t}^i\cup W_{p-j-1,t}^i$, but this distance is some multiple of $p$ due to our construction of $W$ and to Equation (\ref{e:theta}). By moving vertically in $R$ from $W_{j,t}^{i+1}\cup W_{p-j-1,t}^{i+1}$ by some appropriate multiple of $p$ units (possibly zero), we encounter a set $\tilde W_{j,t}^{i+1}\cup \tilde W_{p-j-1,t}^{i+1}$ of $p$ elements in $R$ that is $n/p=p^2$ rows distant from $W_{j,t}^i\cup W_{p-j-1,t}^i$:
$$
\begin{array}{cccc}
W_{j,t}^{i+1} & \ & \ & W_{p-1-j,t}^{i+1} \\
\cdot \cdot \cdot \cdot  & \ & \ & \cdot \cdot \cdot \cdot \cdot \cdot \cdot \\
\downarrow        & \qquad & \qquad & \downarrow \\
\cdot \cdot \cdot \cdot & \ & \ & \cdot \cdot \cdot \cdot \cdot \cdot \cdot  \\
\tilde W_{j,t}^{i+1} & \ & \ & \tilde W_{p-1-j,t}^{i+1}
\end{array}
   $$
Further, by applying Lemma \ref{l:moremoresums2}, we have
$$
w_{j,t}^{i+1}+w_{p-j-1,t}^{i+1}=\tilde w_{j,t}^{i+1}+ \tilde w_{p-j-1,t}^{i+1},
  $$
where $\tilde w_{j,t}^{i+1}$ is the sum of the elements in $\tilde W_{j,t}^{i+1}$, and likewise for $\tilde w_{p-j-1,t}^{i+1}$. The vertical nature of this replacement has no effect on the relationship among columns: it is still true that an element in
$W_{j,t}^i\cup W_{p-j-1,t}^i$ and its counterpart in $\tilde W_{j,t}^{i+1}\cup \tilde W_{p-j-1,t}^{i+1}$ are $n/p=p^2$ columns distant from one another. These statements are also true if ``$t$" is replaced by ``$b$".  By making these replacements systematically and judiciously, so as to avoid repetition of rows, we may apply Lemma \ref{l:moremoresums2} together with the complementary property in $R$ to obtain
\begin{equation}\label{e:sj}
\begin{split}
s_j&=\sum_{i=0}^{p-1} w_j^i+w_{p-1-j}^i \\
&=\sum_{i=0}^{p-1} (w_{j,t}^i +w_{p-1-j,t}^i) +\sum_{i=0}^{p-1} (w_{j,b}^i+w_{p-1-j,b}^i) \\
&=\sum_{i=0}^{p-1} (\tilde w_{j,t}^i +\tilde w_{p-1-j,t}^i) +\sum_{i=0}^{p-1} (\tilde w_{j,b}^i+\tilde w_{p-1-j,b}^i) \\
&=p\left[\frac{p(p^6-1)}{2}\right] +p\left[\frac{p(p^6-1)}{2}\right]
=p^2(p^6-1),
\end{split}
   \end{equation}
thereby proving the first portion of our claim on the sums $s_j$.

  Finally, we address the claimed value of $s_{\frac{p-1}{2}}$. Without loss of generality we assume that $B^{0,0}_{\frac{p-1}{2}}=B^{0,1}_{\frac{p-1}{2}}$. For each $1\leq i \leq p-1$, we may use Lemma \ref{l:moremoresums2} to
  consider elements $\tilde W_{\frac{p-1}{2}}^{i,0}\cup \tilde W_{\frac{p-1}{2}}^{i,1}$ lying above
    $W_{\frac{p-1}{2}}^{i,0}\cup W_{\frac{p-1}{2}}^{i,1}$ and in the same row as $W^{0,0}_{\frac{p-1}{2}}$ as illustrated here:
$$
\begin{array}{ccccc}
\overbrace{\ast \ast \ast}^{ \tilde W^{2,0}_{\frac{p-1}{2}}} & \overbrace{\ast \ast\ast \ast}^{\tilde W^{1,0}_{\frac{p-1}{2}}} & \overbrace{\ast\ast\ast\ast\ast\ast \ast \ast}^{ W^{0,0}_{\frac{p-1}{2}}} & \overbrace{\ast\ast\ast}^{\tilde W^{1,1}_{\frac{p-1}{2}}} & \overbrace{\ast\ast\ast\ast}^{\tilde W^{2,1}_{\frac{p-1}{2}}} \\
\uparrow &     \uparrow &              \                  & \uparrow & \uparrow \\
\ & \ast\ast\ast\ast &  \                  &\ast\ast\ast & \  \\
 \uparrow &  W^{1,0}_{\frac{p-1}{2}}&              \                    & W^{1,1}_{\frac{p-1}{2}} & \uparrow \\
\ast\ast\ast & \     &     \               & \ &\ast\ast\ast\ast \\
 W^{2,0}_{\frac{p-1}{2}}   & \ & \     &\ & W^{2,1}_{\frac{p-1}{2}}
\end{array}
   $$
 If we let
   $\tilde w_{\frac{p-1}{2}}^{i,0} + \tilde w_{\frac{p-1}{2}}^{i,1}$ be the corresponding sum of elements, we  find by applying the $1/p$ row property of $\theta (R)$ (Proposition \ref{p:pcolprop}) that
   \begin{equation}\label{e:sp}
   \begin{split}
   s_{\frac{p-1}{2}}&=
   w_{\frac{p-1}{2}}^{0,0}+\sum_{i=1}^{p-1} \left( w_{\frac{p-1}{2}}^{i,0}+w_{\frac{p-1}{2}}^{i,1}\right)\\
   &=w_{\frac{p-1}{2}}^{0,0}+\sum_{i=1}^{p-1} \left( \tilde w_{\frac{p-1}{2}}^{i,0}+\tilde w_{\frac{p-1}{2}}^{i,1}\right)\\
   &=\frac{1}{p} \left[\frac{n(n^2-1)}{2}\right]
   =\frac{1}{p}\left[\frac{p^3(p^6-1)}{2}\right]
   =\frac{p^2(p^6-1)}{2},
   \end{split}
      \end{equation}
 as claimed. The other Franklin pattern categories (right, down, and left) have similar verifications.

\end{proof}

\section{Type-$p$ Franklin Squares of Order $kp^3$ with $k>1$.}\label{s:greater}
In this section we indicate how type-$p$ Franklin squares of order $kp^3$ can be defined, and argue that these squares exist when $k=p^r$ for $r\geq 0$. This extends the results of Sections \ref{s:bent} and \ref{s:cube}, where we addresed the special case $k=1$. Terminology and ideas of
Sections \ref{s:bent} and \ref{s:cube} will be used throughout.

The description in Section \ref{ss:dubya} characterizes type-$p$ Franklin squares of order $kp^3$ except for the Franklin patterns. As in Section \ref{s:bent}, we focus on describing Franklin-up patterns; the other varieties (right, down, and left) are obtained from Franklin-up locations by rotating the ambient square. Let $S$ be a square of order $n=kp^3$, let $\alpha +\beta =p$ with $1\leq \alpha ,\beta <p$, and let $W$ be a Franklin-up pattern in $S$. The frame for $W$ consists of $\frac{n}{p}=kp^2$ consecutive rows of $S$. As in Section \ref{s:bent}, we can partition this frame into a $p\times p^2$ array $(T_{i,j})$ where $T_{i,j}$ is an array of size $kp\times kp$. Therefore, each of the squares $B_j^i$ and $B_{p-1-j}$ should be of size $kp\times kp$, as should be $B_{\frac{p-1}{2}}^{k,l}$ in case $p$ is odd. To determine $W$ it is necessary to describe the intersection of these squares with $W$.

We first address $W\cap B_j^i$ with $0\leq j <\frac{p-1}{2}$. View $B_j^i$ as a $k\times k$ array whose entries are $p\times p$ subarrays. If $j$ is even, recall that as $i$ increases from $0$ to $p-1$, the squares $B_j^i$ lie on a broken main diagonal of the array $(T_{i,j})$. In this case we declare that $W$ intersects $B_j^i$ in each of the $p \times p$ submatrices on the main block diagonal of $B_j^i$ in the manner described in Section \ref{s:bent} (Figure \ref{f:blockint}). If $j$ is odd, recall that the squares $B_j^i$ lie on a broken off diagonal of the array $(T_{i,j})$. In this case we declare that $W$ intersects $B_j^i$ in each of the $p \times p$ submatrices occupying the off block diagonal of $B_j^i$ in the manner of Section \ref{s:bent}. Intersections of $W$ with $B_{p-1-j}^i$ are determined similarly. A figure illustrating $W\cap B_j^i$ with $j$ even is shown below, where the smaller arrays along the main diagonal are of size $p\times p$.
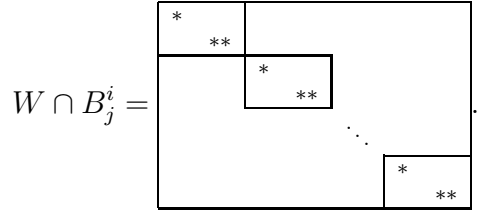
$$
W\cap B_j^i=
{\tiny \begin{array}{|ccccccc|}
\hline
\ast & \multicolumn{1}{c|}{\ }& & & & &  \\
     &\multicolumn{1}{c|}{\ast \ast} & & & & &\\\cline{1-4}
     & \multicolumn{1}{c|}{\ } &\ast &\multicolumn{1}{c|}{\ } & & &  \\
     &\multicolumn{1}{c|}{\ } & &\multicolumn{1}{c|}{\ast \ast} & & &\\\cline{3-4}
     & & &          &\ddots & & \\\cline{6-7}
     & & &          &   \multicolumn{1}{c|}{\ }    &\ast & \\
     & & &          &   \multicolumn{1}{c|}{\ }    &     &\ast \ast\\
     \hline
\end{array}}.
   $$
Also, here is a frame showing all  blocks $B_j^i$ in the classical case $p=2$ and $n=2\cdot 2^3=16$:
$${\tiny
\begin{array}{|cccc|cccc|cccc|cccc|}
\hline
\ast &\multicolumn{1}{c|}{} & & & & & & & & & & & &\multicolumn{1}{c|}{}  & &  \ast \\
    & \multicolumn{1}{c|}{\ast} & & & &\leftarrow & B_0^0& & &B_1^0 &\rightarrow & & &\multicolumn{1}{c|}{}  &\ast &   \\\cline{1-4}\cline{13-16}
    & \multicolumn{1}{c|}{}& \ast  & & & & & & & & & & & \multicolumn{1}{c|}{\ast} & &  \\
   & \multicolumn{1}{c|}{}& & \ast & & & & & & & & &\ast &\multicolumn{1}{c|}{}  & &   \\
\hline
    & & & &\ast & \multicolumn{1}{c|}{} & & & &\multicolumn{1}{c|}{}  & & \ast & & & &   \\
   &B_0^1 & \rightarrow & & &\multicolumn{1}{c|}{\ast} & & & &\multicolumn{1}{c|}{}  &\ast & & &\leftarrow &B_1^1 &   \\\cline{5-12}
    & & & & &\multicolumn{1}{c|}{}  &\ast & & &\multicolumn{1}{c|}{\ast} &  & & & & &   \\
    & & & & &\multicolumn{1}{c|}{}  & &\ast &\ast &\multicolumn{1}{c|}{}  &  & & & & &   \\
\hline
\end{array}}
   $$

It remains to address the intersection of the Franklin-up pattern $W$ with the middle band $B_{\frac{p-1}{2}}$ in the case that $p$ is odd. Unlike the other bands, we will continue to partition
$B_{\frac{p-1}{2}}$ into $p\times p$ subsquares as we did in Section \ref{s:bent}. (This is reasonable because we do not apply $\theta $ to this band in Theorem \ref{t:MPtoF2}, and so we do not need a partition into squares of order $\frac{n}{p^2}=kp$.) Further, we define $B_{\frac{p-1}{2}}^{i,0}$ and $B_{\frac{p-1}{2}}^{i,1}$, as well as their intersections with $W$ just as we did in Section \ref{s:bent}, except that $0\leq i\leq kp-1$ rather than $0\leq i\leq p-1$ (Figure \ref{f:blockint2}). We note that in the special case that $B_{\frac{p-1}{2}}^{i,0}$ and $B_{\frac{p-1}{2}}^{i,1}$ coincide, then the intersection with $W$ is the entire bottom row of this square; this will happen when $k$ is odd. Meanwhile, in the special case that $B_{\frac{p-1}{2}}^{i,0}$ and $B_{\frac{p-1}{2}}^{i,1}$ are adjacent (borders touching) then their intersection with $W$ consists of the entire bottom row of both squares. This latter case, which happens when $k$ is even, produces a row in the frame for $W$ that intersects $W$ in $2p$ locations rather than $p$ locations. An illustration is given in the following figure, which shows the middle band $B_1$ in the case $n=kp^3=3\cdot 3^3$. Each entry is a $3\times 3$ array; the asterisks are the $B_1^{i,\ell }$'s. The boxed asterisk is $B_1^{8,1}$; its intersection with $W$ is shown in the right portion of the figure (assuming $\alpha =1$ and $\beta =2$).

$$
{\tiny
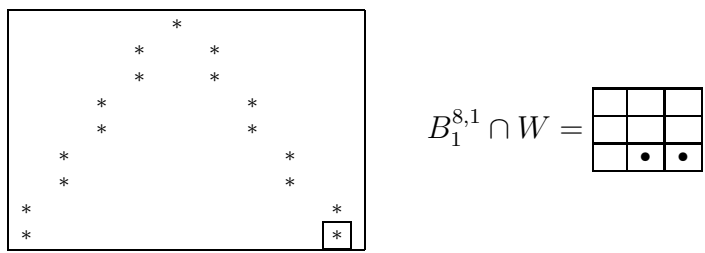
\begin{array}{|ccccccccc|}
\hline
  & & & &\ast & & & &   \\
  & & &\ast & &\ast & & &   \\
  & & &\ast & &\ast & & &   \\
  & &\ast & & & &\ast & &   \\
  & &\ast & & & &\ast & &   \\
  &\ast & & & & & & \ast &   \\
  & \ast& & & & & &\ast &   \\
 \ast & & & & & & & & \ast  \\
 \ast & & & & & & & & \boxed{\ast}  \\
 \hline
\end{array}}
\qquad
B_1^{8,1}\cap W=
{\tiny
\begin{array}{|c|c|c|}
\hline
 & & \\
 \hline
 & & \\
 \hline
\  &\bullet & \bullet \\
 \hline
\end{array}
}
   $$

\begin{thm}\label{t:MPtoF2}
Let $p$ be prime, $k\in \Z^+$, and $n=kp^3$. If $R$ is a type-$p$ most-perfect square of order $n$ then $\theta (R)$ is an order-$n$ pandiagonal type-$p$ Franklin square. Further, such squares $R$ exist when $k=p^r$ for any prime $p$ and any $r\geq 0$.
\end{thm}

\begin{proof}
The proof, which shall be abridged, closely follows that for Theorem \ref{t:k1}. Notation will be identical to that of Theorem \ref{t:k1}, with the exception that
$w_j^i$ will be split into $2k$ summands rather than just two summands $w_{j,t}^i$ and $w_{j,b}^i$. This is due to the fact that $W$ intersects $B_j^i$ in $2k$ rows rather than $2$ rows. (A similar adjustment is made for $w_{p-1-j}^i$.)

Let $S=\theta (R)$. Due to Proposition \ref{p:culminate}, to establish that $S$ is a type-$p$ Franklin square it remains to show that entries in Franklin patterns add to the magic sum. We verify this for Franklin-up patterns only, the other patterns have similar verifications. Following the proof of Theorem \ref{t:k1}, and Equation (\ref{e:sj}) in particular, the use of Lemma \ref{l:moremoresums2} and the complementary property in $R$ gives
$$
s_j=\sum_{i=0}^{p-1}w_j^i+w_{p-1-j}^i=\underbrace{p\left[\frac{p(n^2-1)}{2}\right]+\cdots +p\left[\frac{p(n^2-1)}{2}\right]}_{2k\ {\rm times}}=\frac{n(n^2-1)}{p}
  $$
when $0\leq j < \frac{p-1}{2}$. Likewise, in the case that $p$ is odd, applying Lemma \ref{l:moremoresums2} together with the $1/p$-row property of $S$ as in
Equation (\ref{e:sp}) gives
$s_{\frac{p-1}{2}}=\frac{n(n^2-1)}{2p}$. It follows that the sum of the entries in W is
$$
\sum_{0\leq j\leq \frac{p-1}{2}} s_j=\frac{n(n^2-1)}{2},
  $$
as desired.

Finally,
the existence of  type-$p$ most-perfect squares of order $p^s$ ($s\geq 3$) is guaranteed by \cite{jL18}.

\end{proof}

\end{document}